\documentclass[a4paper,11pt]{amsart}

\let\oldtocsection=\tocsection
 
\let\oldtocsubsection=\tocsubsection
 
\let\oldtocsubsubsection=\tocsubsubsection
 
\renewcommand{\tocsection}[2]{\hspace{0em}\oldtocsection{#1}{#2}\bfseries}
\renewcommand{\tocsubsection}[2]{\hspace{1.8em}\oldtocsubsection{#1}{#2}}
\renewcommand{\tocsubsubsection}[2]{\hspace{4.4em}\oldtocsubsubsection{#1}{#2}}

\makeatletter
\renewcommand\subsection{\@startsection{subsection}{2}%
  \z@{-.5\linespacing\@plus-.7\linespacing}{.5\linespacing}%
  {\normalfont\scshape}}
\renewcommand\subsubsection{\@startsection{subsubsection}{3}%
  \z@{.5\linespacing\@plus.7\linespacing}{.5\linespacing}%
  {\normalfont\scshape}}
\makeatother
\usepackage[utf8]{inputenc}
\usepackage[T1]{fontenc}
\usepackage{tikz-cd}
\usepackage{paralist}
\usepackage{mathrsfs}
\usepackage{amssymb}
\usepackage{amsmath}
\usepackage{amsfonts}

\newtheorem{theorem}{Theorem}[section]

\newtheorem{lemma}[theorem]{Lemma}

\newtheorem{question}[theorem]{Question}

\newtheorem{remark}[theorem]{Remark}
\newtheorem{Fact}[theorem]{Fact}
\newtheorem{Claim}[theorem]{Claim}

\theoremstyle{definition}
\newtheorem{definition}[theorem]{Definition}

\usepackage[pdftex]{hyperref}
\hypersetup{
    colorlinks=true, %set true if you want colored links
    linktoc=all,     %set to all if you want both sections and subsections linked
    linkcolor=blue,  %choose some color if you want links to stand out
    allcolors=blue,
}

\newcommand{\dom}[1]{\ensuremath{\mathrm{dom}}(#1)}

\newcommand{\set}[2]{\ensuremath{\{#1 \,|\, #2 \}}}
\newcommand{\seq}[2]{\ensuremath{\langle #1 \,|\, #2 \rangle}}
\newcommand{\restr}[2]{\ensuremath{#1 \! \upharpoonright \! #2}}
\renewcommand{\iff}{\leftrightarrow}

\newcommand{\sub}{\subseteq}

\newcommand{\then}{\rightarrow}

\newcommand{\bb}{\mathbb}

\newcommand{\beq}{\begin{equation}}
\newcommand{\eeq}{\end{equation}}
\newcommand{\brm}{\begin{remark}\begin{rm}}
\newcommand{\erm}{\end{rm}\end{remark}}
\newcommand{\mx}{\mathrm}
\newcommand{\bce}{\begin{compactenum}}
\newcommand{\ece}{\end{compactenum}}
\newcommand{\et}{\mathrel{\&}}

\newcommand{\Ult}{\mathrm{Ult}}

\newcommand{\cof}{\mathrm{cof}}
\newcommand{\cf}{\mathrm{cf}}
\newcommand{\tcf}{\mathrm{tcf}}
\newcommand{\lh}{\mathrm{lh}}

\newcommand{\Add}{\mathrm{Add}}
\newcommand{\Sacks}{\mathrm{Sacks}}

\newcommand{\R}{\bb{R}}
\newcommand{\Q}{\bb{Q}}
\renewcommand{\P}{\bb{P}}

\newcommand{\T}{\bb{T}}
\newcommand{\M}{\bb{M}}
\newcommand{\x}{\times}

\newcommand{\TP}{{\sf TP}}

\newcommand{\ZFC}{\sf ZFC}

\newcommand{\GCH}{\sf GCH}
\newcommand{\SCH}{\sf SCH}

\newcommand{\SR}{{\sf SR}}
\newcommand{\AP}{{\sf AP}}

\newcommand{\uu}{\mathfrak{u}(\kappa)}

\newcommand{\Jbd}{J^{\mx{bd}}_\kappa}
\newcommand{\Jbdpar}[1]{J^{\mx{bd}}_{#1}}
\newcommand{\Qt}{\Q_{\bar{\theta}}}
\renewcommand{\Upsilon}{\lambda}
\newcommand{\Ups}{\lambda}

\newcommand{\Pm}{\R^{\mathrm{Mag}}}
\newcommand{\Qm}{\Q^{\mathrm{Mag}}_{\vec{U}}}
\newcommand{\Ref}{\mathrm{Ref}}
\newcommand{\CSR}{{\sf CSR}}
\newcommand{\Ah}{A^{\mathrm{hom}}}

\begin{document}

\title{Small $\uu$ at singular $\kappa$ with compactness at $\kappa^{++}$}

\author{Radek Honzik}
\address[Honzik]{
Charles University, Department of Logic,
Celetn{\' a} 20, Prague~1, 
116 42, Czech Republic
}
\email{radek.honzik@ff.cuni.cz}
\urladdr{logika.ff.cuni.cz/radek}

\author{{\v S}{\'a}rka Stejskalov{\'a}}
\address[Stejskalov{\'a}]{
Charles University, Department of Logic,
Celetn{\' a} 20, Prague~1, 
116 42, Czech Republic
}
\email{sarka.stejskalova@ff.cuni.cz}
\urladdr{logika.ff.cuni.cz/sarka}

\thanks{
Both authors were supported by FWF/GA{\v C}R grant \emph{Compactness principles and combinatorics} (19-29633L)}

\begin{abstract}
We show that the tree property, stationary reflection and the failure of approachability at $\kappa^{++}$ are consistent with $\uu = \kappa^+ < 2^\kappa$, where $\kappa$ is a singular strong limit cardinal with the countable or uncountable cofinality. As a by-product, we show that if $\lambda$ is a regular cardinal, then stationary reflection at $\lambda^+$ is indestructible under all $\lambda$-cc forcings (out of general interest, we also state a related result for the preservation of club stationary reflection).
\end{abstract}

	\maketitle

\tableofcontents

\section{Introduction}\label{sec:intro}

Compactness properties at successor cardinals have been recently extensively studied, with the focus on the tree property, stationary reflection and the failure of approachability (see Section \ref{sec:prelim-c} for definitions). The underlying goal is to find out whether these principles can hold at the successor or the double successor of a singular cardinal, a long interval of cardinals, or whether they determine the continuum function in some non-trivial way. In our paper we investigate these properties from yet another direction and study their compatibility with a small ultrafilter number $\uu$ (where ``small'' means smaller than $2^\kappa$).

Recall that $\uu$ is the least cardinal $\alpha$ such that there exists a base $B$ of a uniform ultrafilter $U$ on $\kappa$ of size $\alpha$ ($U$ is uniform if every $X \in U$ has size $\kappa$ and $B$ is a base of a uniform ultrafilter $U$ if for every $X \in U$ there is $Y \in B$ with $Y \sub X$). For all infinite $\kappa$, $\uu$ is always at least $\kappa^+$ by \cite{GS:small}. It is of interest to study whether $\uu <2^\kappa$ is consistent for various $\kappa$. By \cite{GRIG:ideals}, the consistency of $\ZFC$ is sufficient for $\mathfrak{u}(\omega) = \omega_1$ with $2^\omega = \omega_2$ as it can be obtained by the iteration (and also a product) of the Sacks forcing up to $\omega_2$. For inaccessible $\kappa$, it is not clear whether a small $\uu$ has any large cardinal strength, but supercompact cardinals are used in the known constructions (see for instance \cite{GS:pol,F:u}). For a strong limit singular $\kappa$, more information is known: by a recent result in \cite{GGS:u}, at a strong limit $\aleph_\omega$ the consistency strength of $\mathfrak{u}(\aleph_\omega) < 2^{\aleph_\omega}$ is exactly that of $\aleph_{\omega+1}<2^{\aleph_\omega}$ (it is hard to speculate what it says about a small $\uu$ regarding a lower bound for a regular $\kappa$). By a recent result in \cite{RS:u}, it is consistent that $\uu$ is small for a weakly inaccessible $\kappa$ (in fact for $\kappa = 2^\omega)$ and for a successor of a singular cardinal (more specifically, $\aleph_{\omega+1}$). It is open whether $\uu < 2^\kappa$ is consistent for a successor of a regular cardinal.

The ultrafilter number $\uu$ is one of the \emph{generalized cardinal invariants} which study the combinatorial properties of the spaces $\kappa^\kappa$ or $2^\kappa$ for topological, purely combinatorial, or forcing-related reasons. Since the tree property and the failure of approachability at $\kappa^{++}$ both imply $2^\kappa>\kappa^+$, they make the structure of the generalized cardinal invariants at $\kappa$ possibly non-trivial. It is natural to ask to what extent the invariants can be manipulated while ensuring compactness at $\kappa^{++}$.

For $\kappa = \omega$, this problem is easier to grasp because the cardinal invariants at $\omega$ have been studied for some time, as have been the forcings to force the tree property and other principles at $\omega_2$.\footnote{However, to our knowledge there is no publication which surveys which cardinal invariants patterns at $\omega$ can be realized with compactness at $\omega_2$.} Let us list just a few examples, focusing only on the tree property and stationary reflection for brevity. The iteration of the Sacks forcing at $\omega$ up to a weakly compact cardinal $\lambda$ forces $2^\omega = \omega_2 = \lambda$, the tree property and stationary reflection at $\omega_2$, and keeps most of the cardinal invariants at $\omega_1$ (in particular $\mathfrak{u}(\omega) = \omega_1$ by \cite{GRIG:ideals}). The Mitchell forcing iterated up to a weakly compact $\lambda$ forces $2^\omega = \omega_2 = \lambda$, the tree property and stationary reflection at $\omega_2$ and $\mathfrak{u}(\omega) = \omega_2$. A proper forcing iteration from \cite{FT:tp} iterated up to a weakly compact and reflecting $\lambda$ forces ${\sf BPFA}$ with $2^\omega = \omega_2 = \lambda$, the tree property and stationary reflection at $\omega_2$, with the majority of the cardinal invariants at $\omega_2$ (including $\mathfrak{u}(\omega)$).

For a regular $\kappa>\omega$, the situation is less studied. It is open how to generalize the method from \cite{FT:tp} because it is based on the notion of proper forcing. The Mitchell forcing and the Sacks forcings can still be used to obtain the tree property and stationary reflection at $\kappa^{++}$, but the invariants in these generic extensions are not automatically the same as in the case of $\kappa = \omega$ -- for instance it is not known whether the generalized Sacks iteration $\Sacks(\kappa,\lambda)$ forces a small $\uu$ (the argument from \cite{GRIG:ideals} fails because the perfect trees now have limit levels). In our paper under preparation \cite{vFHS:u} we use a variant of the Mathias forcing iteration based on \cite{F:u} to obtain a model where $\kappa>\omega$ is regular (in fact supercompact), $2^\kappa$ is arbitrarily large, $\uu$ is small and the compactness principles hold at $\kappa^{++}$.

In this paper, we focus on  $\kappa$ which is a strong limit singular cardinal. The possibility of having $\kappa$ singular is even more interesting from the point of compactness at $\kappa^{++}$: it combines three intriguing properties -- the necessary failure of $\SCH$ at $\kappa$, compactness at $\kappa^{++}$ and non-trivial cardinal invariants at $\kappa$. 

The paper is structured as follows. In Section \ref{sec:prelim} we review some basic definitions and facts a provide a very brief introduction to pcf notions which appear in the proof. 

In Section \ref{sec:review}, we review the original argument of Garti and Shelah from \cite{GS:pol,GS:small} which yields a model with a singular strong limit $\kappa$ which violates $\SCH$ and $\uu = \kappa^+$.

In Section \ref{sec:countable}, we use the forcing from Garti and Shelah to define a variant of the Mitchell forcing followed by the Prikry forcing which in addition to $\uu = \kappa^+$ forces also the tree property, stationary reflection and the failure of approachability at $\kappa^{++}$. All three arguments use a form of indestructibility of the relevant compactness property by certain $\kappa^+$-cc forcings. The argument for the tree property in Section \ref{sec:tp} is based on the fact that over a certain type of models (such as the one we consider here), the tree property is indestructible under a wide class of $\kappa^+$-cc forcings, in particular under the Prikry forcing; the argument is based on the indestructibility result which appears in \cite{HS:ind}. The arguments for stationary reflection and the failure of approachability in Sections \ref{sec:sr} and \ref{sec:ap}, respectively, use stronger forms of indestructibility which hold over any model. We show that for a regular $\lambda$, stationary reflection at $\lambda^+$ is indestructible under all $\lambda$-cc forcings (this improves the existing results; see Theorem \ref{th:sr}). As a matter of general interest, we also show that a stronger form of stationary reflection (club stationary reflection) is preserved by Cohen forcings (and some other forcings); see Theorem \ref{th:csr}. For  the failure of approachability, we use a result from \cite{GK:a} that $\neg \AP(\kappa^{++})$ is preserved under all $\kappa$-centered forcings.

All three arguments have the advantage of offering a direct generalization to other Prikry-like forcings such as Magidor forcing: in Section \ref{sec:uncountable} we show that $\uu = \kappa^+$ can also be obtained with compactness at $\kappa^{++}$ for a strong limit singular $\kappa$ of uncountable cofinality.

In Section \ref{sec:open} we discuss some open questions.

\section{Preliminaries}\label{sec:prelim}

Let us review some definitions and facts which will be used in the proofs.

\subsection{Compactness properties}\label{sec:prelim-c}

\begin{definition}
Let $\lambda$ be a regular cardinal. We say that the \emph{tree property} holds at $\lambda$, and we write $\TP(\lambda)$, if every $\lambda$-tree has a cofinal branch.
\end{definition}

\begin{definition}
Let $\lambda$ be a cardinal of the form $\lambda = \nu^{+}$ for some regular cardinal $\nu$. We say that the \emph{stationary reflection} holds at $\lambda$, and write $\SR(\lambda)$, if every stationary subset $S \sub \lambda \cap \cof(< \nu)$ reflects at a point of cofinality $\nu$; i.e.\ there is $\alpha < \lambda$ of cofinality $\nu$ such that $\alpha \cap S$ is stationary in $\alpha$.
\end{definition}

For a cardinal $\lambda$ and sequence $\bar{a}=\seq{a_\alpha}{\alpha<\lambda}$ of bounded subsets of $\lambda$, we say that an ordinal $\gamma<\lambda$ is approachable with respect to $\bar{a}$ if there is an unbounded subset $A\sub\gamma$ of order type $\cf(\gamma)$ and for  all $\beta<\gamma$ there is $\alpha<\gamma$ such that $A\cap\beta=a_\alpha$.

Let us define the \emph{ideal $I[\lambda]$ of approachable subsets of $\lambda$}:

\begin{definition}
$S\in I[\lambda]$ if and only if there are a sequence $\bar{a} = \seq{a_\alpha}{\alpha<\lambda}$ and a club $C\sub \lambda$ such that every $\gamma\in S\cap C$ is approachable with respect to $\bar{a}$.
\end{definition}

\begin{definition}
We say that \emph{the approachability property holds} at $\lambda$ if $\lambda\in I[\lambda]$, and we write $\AP(\lambda)$. If $\neg \AP(\lambda)$, we say that \emph{approachability fails at $\lambda$.}
\end{definition}

It is known that $\neg \AP(\lambda)$, for $\lambda = \nu^+$, implies the failure of $\square^*_\nu$, and in this sense $\neg \AP(\lambda)$ can be considered a compactness principle. Note that $\AP(\lambda)$ does not imply $\square^*_\nu$, so $\neg \AP(\lambda)$ is strictly stronger than the fact that there are no special $\lambda$-Aronszajn trees.

\subsection{Branch lemmas}

The so called ``branch lemma'' are used to argue that certain forcings do not add cofinal branches to trees.

The first fact is sometimes attributed to Kurepa and is stated in Kunen's book \cite{KUNbook}, Exercise V.4.21. Recall that a $\kappa^+$-tree $T$ is called \emph{well-pruned} if it has a single root and for every node $t \in T$ and level $\alpha$ above the level of $t$, there is a node $t' \in T$ on level $\alpha$ which is above $t$ in the ordering of $T$.

\begin{Fact}\label{Kunen} 
The following hold:
\bce[(i)]
\item Suppose $T$ is well-pruned $\kappa^+$-Aronszajn tree. Then for every $t \in T$, there is a level of the tree $T$ above $t$ which has size $\kappa$.
\item It follows that if $P$ is $\kappa$-cc, then it does not add a cofinal branch to $T$.
\ece
\end{Fact}

\begin{proof}
(i) For contradiction assume that all levels above $t$ have size $<\kappa$, and using Fodor's lemma, find a stationary set on which the nodes of the tree form a cofinal branch.

(ii) If $\dot{b}$ is a name for a cofinal branch through $T$, it can be used to build back in $V$ a subtree $S$ of $T$ of height $\kappa^+$ with levels of size $<\kappa$. By (i), $S$ must have a cofinal branch, and it is a cofinal branch through $T$ as well. Contradiction.
\end{proof}

Fact \ref{Kunen} is useful for showing the indestructibility result in \cite{HS:ind}, which we apply in this paper in a different setting.

We shall further use the following lemma due to Unger (see \cite[Lemma 6]{UNGER:1}), which generalizes an analogous result in \cite{JSkurepa} which is formulated for $\kappa = \omega$:

\begin{Fact}\label{f:spencer}
Let $\kappa,\lambda$ be cardinals with $\lambda$ regular and $\kappa < \lambda \le 2^\kappa$. Let $P$ be $\kappa^+$-cc and $Q$ be $\kappa^+$-closed. Let $\dot{T}$ be a $P$-name for a $\lambda$-tree. Then in $V[P]$, forcing with $Q$ cannot add a cofinal branch through $T$.
\end{Fact}

\subsection{Some pcf notions}

Let us briefly review pcf concepts which appear in the arguments below. Suppose $A$ is a set of cardinals, typically with $|A| < \mx{min}(A)$, and $J$ is an \emph{ideal} on $A$. We denote by $\prod A$ the collection of all functions $f: A \to \bigcup A$ such that for every $\alpha \in A$, $f(\alpha) \in \alpha$. For $f,g \in \prod A$, we write \beq f <_J g  \iff \set{\alpha \in A}{f(\alpha) \ge g(\alpha)} \in J.\eeq The relation $\le_J$ and $=_J$ is defined similarly (i.e.\ the set of counterexamples is in $J$). Notice that \beq \label{eq:<} (f <_J g \vee f =_J g) \then f \le_J g,\eeq but the converse is in general true only if $J$ is a prime ideal (it may happen that $f \le_J g$ and the sets $B_1 = \set{\alpha \in A}{f(\alpha) = g(\alpha)}$ and $B_2 = \set{\alpha \in A}{f(\alpha) \neq g(\alpha)}$ are both $J$-positive and disjoint). Some weaker relationships between $<_J$ and $\le_J$ are useful, in particular this one: \beq \label{eq:weak} (f <_J g \et g \le_J h) \then f <_J h.\eeq

\begin{definition}
We say that the ordered set $(\prod A, <_J)$ has \emph{cofinality} $\kappa$, and write $\cf(\prod A, <_J) = \kappa$, if $\kappa$ is the least cardinal such that there exists $X \sub \prod A$ of size $\kappa$ which is cofinal in $\prod A$ under $<_J$, i.e.\ for every $f \in \prod A$ there is $g \in X$ with $f <_J g$ or $f = g$.\footnote{This is strictly stronger than requiring $\le_J$ by (\ref{eq:<}). Note that if $\prod A$ has no maximal elements, then the clause $f = g$ can be omitted.}
\end{definition}

Notice that it is not required that the members of the cofinal set $X$ are themselves ordered by $<_J$ (this will lead to the notion of \emph{true cofinality} introduced below).

We say that $(\prod A, <_J)$ is $\kappa$-\emph{directed (closed)} if every subset $X$ of $\prod A$ of size $<\kappa$ has a $<_J$-upper bound.

\begin{definition}
We say that $(\prod A, <_J)$ has \emph{true cofinality} $\kappa$, and write $\tcf(\prod A, <_J) = \kappa$, if there is a $<_J$-increasing sequence $\seq{f_i}{i<\kappa}$ such that for every $g \in \prod A$ there is some $i$ with $g <_J f_i$, and $\kappa$ is least such cardinal.
\end{definition}

It is clear that any $\seq{f_i}{i<\kappa}$ which is a witness for the true cofinality is also a cofinal subset of $(\prod A,<_J)$ and therefore $\cf(\prod A, <_J) \le \tcf(\prod A, <_J)$. It also holds that if $\tcf(\prod A, <_J) = \kappa$, then $(\prod A,<_J)$ is $\kappa$-directed closed.

\subsection{Prikry and Magidor forcing}\label{sec:Prk}

For completeness, let us review the definitions and basic properties of forcings which singularize a large cardinal $\kappa$ by adding a cofinal sequence of order type $\alpha<\kappa$.

Recall that $\Q$ is \emph{$\kappa$-centered} for a regular $\kappa$ if $\Q$ can be written as the union of the family $\set{\Q_\alpha \sub \Q}{\alpha <\kappa}$ such that for every $\alpha<\kappa$ and every $p,q \in \Q_\alpha$ there exists $r \in \Q_\alpha$ with $r \le p,q$.

\begin{definition}\label{def:Prikry}
Let $\kappa$ be a measurable cardinal and $U$ a normal ultrafilter on $\kappa$. Prikry forcing $\Q_U$ is composed of pairs $(s,A)$, where $s$ is a finite set of ordinals below $\kappa$, $A$ is in $U$ and $\mx{max}(s) + 1 \cap A = \emptyset$. $(t,B) \le (s,A)$ iff $t$ is an end-extension of $s$ (which we write $s \sqsubseteq t$), $B \sub A$, and $t \setminus s$ is included in $A$. We call $s$ \emph{a stem}.
\end{definition}

$\Q_U$ adds an $\omega$-sequence which is cofinal in $\kappa$ without collapsing any cardinals. It is easy to see that all conditions with the same stem are compatible. This implies that $\Q_U$ is $\kappa^+$-cc, in fact $\kappa$-centered. 

Magidor \cite{M:cof} formulated a generalization of the Prikry forcing which adds a cofinal sequence to $\kappa$ of cofinality $\omega < \mu<\kappa$ without collapsing any cardinals. The idea is to use $\mu$-many normal measures $\vec{U} = \seq{U_\gamma}{\gamma < \mu}$ on $\kappa$ to add $\mu$-many $\omega$-Prikry sequences which together make a cofinal subset of $\kappa$ of ordertype $\mu$; to avoid unwanted interference between the Prikry sequences, the sequence $\vec{U}$ is coherent in a certain sense, more precisely, for each $\gamma < \mu$, $\seq{U_\beta}{\beta<\gamma}$ is an element of the ultrapower of $V$ via $U_\gamma$, denoted $\Ult(V,U_\gamma)$. We say that $\vec{U}$ is a Mitchell sequence of length $\mu$. The following easy lemma will be useful for us.

\begin{lemma}\label{lm:vec-U}
Assume $\vec{U}$ is a Mitchell sequence on $\kappa$ of length $\mu<\kappa$ and $\Q$ is a $\kappa^+$-distributive forcing notion. Then $\vec{U}$ is a Mitchell sequence on $\kappa$ of length $\mu$ in $V[\Q]$.
\end{lemma}

\begin{proof}
Notice that by the $\kappa^+$-distributivity of $\Q$, every normal measure $U$ on $\kappa$ stays a normal measure in $V[\Q]$. Moreover, if $(j_U)^V: V \to \Ult(V,U)$ and $(j_U)^{V[\Q]}: V[\Q] \to \Ult(V[\Q],U)$ are the canonical elementary embeddings to the respective ultrapowers, then $(j_U)^V(\kappa) = (j_U)^{V[\Q]}(\kappa) = \nu$ and $\Ult(V,U)_\nu = \Ult(V[\Q],U)_\nu$, i.e.\ the ultrapowers agree on the sets in the $V$-hierarchy up to $\nu$. This follows from the fact that all sets of rank less than $\nu$ are expressible as (the transitive collapse) of an equivalence class of some $f: \kappa \to V_\kappa$, and $\Q$ does not add such functions. This implies that for every $\gamma <\mu$, $\seq{U_\beta}{\beta < \gamma}$ -- a set of rank less than $(j_{U_\gamma})^{V[\Q]}(\kappa)$ -- is an element of $\Ult(V[\Q],U_\gamma)$, and hence $\vec{U}$ is a Mitchell sequence in $V[\Q]$.
\end{proof}

Magidor forcing, which we denote $\Qm$, has a more complicated definition (see \cite{M:cof} for details), so let us just review the basic points which are relevant for us.

Conditions in $\Qm$ are pairs $(g,G)$ where $g$ is an increasing function from a finite subset of $\mu$ to $\kappa$ whose range concentrates on (specially chosen) inaccessible cardinals (we call $g$ a stem). $G$ is a set of measure-one constraints. An extension $(g',G')$ of $(g,G)$ is allowed to have new elements in $\dom{g'}$ which lie in-between the elements in $\dom{g}$; for this reason, the constraints in $G$ refer not only to normal measures on $\kappa$, but also to normal measures on cardinals below $\kappa$ (defined using the coherence of $\vec{U}$).  Conditions with the same stems are compatible which implies that $\Qm$ is $\kappa^+$-cc, in fact $\kappa$-centered.

\section{A review of the arguments for a small $\uu$}\label{sec:review}

We will use the arguments in papers \cite{GS:pol,GS:small} to get a small $\uu$, and modify the relevant forcings to yield a model with compactness at $\kappa^{++}$. In \cite{GS:pol}, Theorem 1, and Claim 3.3, a forcing is constructed which yields the following model $V^*$ (we are changing the original notation to fit our notation):

\begin{theorem}[\cite{GS:pol}]
Assume there is a supercompact cardinal $\kappa$ in the ground model and $\delta \ge \kappa^{++}$ is cardinal such that $\kappa < \cf(\delta) \le \delta$. Then in a forcing extension $V^*$, $2^\kappa = \delta$ and $\kappa$ is a singular strong limit cardinal with countable cofinality which is a limit of measurables $\seq{\kappa_i}{i<\omega}$ such that $2^{\kappa_i} = \kappa_i^+$ for every $i<\omega$ and both products \beq \label{eq:prod} \prod_{i<\omega}\kappa_i/\Jbdpar{\omega} \mbox{ and }\prod_{i<\omega}\kappa^+_i/\Jbdpar{\omega} \eeq are $\cf(\delta)$-directed closed, where $\Jbdpar{\mu}$ is the ideal of bounded subsets of a regular cardinal $\mu$.
\end{theorem}

Let us review, and at times attempt to clarify, the basic steps of the proof (we will also state a stronger formulation of the result using true cofinalities which is implicit in \cite{GS:pol}). 

Let $V =  V'[\P^L]$ be a model where a supercompact cardinal $\kappa$ is made indestructible by a suitable Laver-like preparation $\P^L$ over $V'$. We require that \begin{equation} \label{PL} \mbox{$\P^L$ applies to the forcing $\P_\delta$ introduced below,}\end{equation} so that $\kappa$ is still supercompact in $V[\P_\delta]$ (see \cite[Claim 2.2]{GS:pol} for more details). Note that we can assume that $2^\kappa = \kappa^+$ in $V$ (by collapsing $2^\kappa$ to $\kappa^+$ over $V'[\P^L]$ by a $\kappa$-directed closed forcing if necessary) so that by reflection there is an increasing cofinal sequence of measurables $\bar{\kappa} = \seq{\kappa_\alpha}{\alpha<\kappa}$ converging to $\kappa$, with $2^{\kappa_\alpha} = \kappa_\alpha^+$ for each $\alpha <\kappa$ (see \cite[Claim 3.3]{GS:pol} for more details).

Let us work in this $V$. A \emph{$\bar{\theta}$-dominating forcing} is defined in \cite[Definition 2.3]{GS:pol}, which is denoted $\Qt$. Given a supercompact cardinal $\kappa$ and a sequence (with some closure properties) $\bar{\theta}$ of length $\kappa$ of regular cardinals converging to $\kappa$, $\Qt$ has the following properties:
\medskip
\bce[(i)]
\item $\Qt$ is $\kappa$-2-linked (by definition this implies $\kappa^+$-Knaster).
\item $\Qt$ is $<\kappa$-strategically closed.
\item In $V[\Qt]$, $(\prod_{\alpha<\kappa}\theta_\alpha/\Jbd)^V$ is dominated by the generic function $g_{\bar{\theta}}$.
\ece
\medskip
In order to secure the claim in \cite[Lemma 2.11]{GS:pol} which requires cofinality with respect to $<_{\Jbd}$, (iii) should mean that for every $ f\in \prod_{\alpha<\kappa}\theta_\alpha$, $f <_{\Jbd} g_{\bar{\theta}}$ in $V[\Qt]$. This is not stated clearly in \cite[Lemma 2.11]{GS:pol} (the ordering $\le_{\Jbd}$ is mentioned instead), so let us state it explicitly here:

{~}
\begin{lemma}\label{lm:correction1}
For every $ f\in (\prod_{\alpha<\kappa}\theta_\alpha)^V$, $f <_{\Jbd} g_{\bar{\theta}}$ in $V[\Qt]$, where $g_{\bar{\theta}}$ is the generic function added by $\Qt$.
\end{lemma}

\begin{proof}
For each $f \in (\prod_{\alpha<\kappa}\theta_\alpha)^V$, let \beq D_f = \set{(\eta,f') \in \Qt}{(\forall \varepsilon \in [\lh(\eta),\kappa)) \; f(\varepsilon) < f'(\varepsilon)},\eeq where the notation is as in \cite{GS:pol}. It is easy to see that $D_f$ is dense. It follows $g_{\bar{\theta}}$ is strictly greater on an end segment of $\kappa$, and so $f <_{\Jbd} g_{\bar{\theta}}$ in $V[\Qt]$.
\end{proof}

In \cite[Definition 2.6]{GS:pol}, $\Qt$ is iterated for length $\delta$ with $<\kappa$-support; the iteration is denoted $\P_\delta$. Some bookkeeping mechanism (or lottery) is used to choose the $\bar{\theta}$'s so that each possible $\bar{\theta}$ appears at some stage of the iteration.\footnote{In fact, only two specific sequences in $V$ are actually required for the main theorem (the sequence of measurables $\bar{\kappa}$ satisfying $\GCH$ and its successors; see below).} \cite[Lemma 2.9]{GS:pol} shows that $\P_\delta$ is $\kappa$-2-linked, and consequently $\kappa^+$-Knaster. The key Lemma 2.11 claims that for a fixed $\bar{\theta}$, the cofinality of $(\prod_{\alpha<\kappa}\theta_\alpha,<_{\Jbd})^{V[\P_\delta]}$ is $\cf(\delta)$ in $V[\P_\delta]$; this claim is proved by showing that the generic functions $\seq{g_{\bar{\theta}_\alpha}}{\alpha \in I}$ for some cofinal set $I$ in $\delta$ of order-type $\cf(\delta)$ are $\le_{\Jbd}$-increasing and bound functions added previously.\footnote{The proof of this abuses the notation and treats $\P_\delta$ as defined with respect to this single $\bar{\theta}$; to handle all the relevant $\bar{\theta}$'s, we need to ensure that each $\bar{\theta}$ appears cofinally often below $\delta$ in $\P_\delta$.} This can be stated more strongly as follows:

\begin{lemma}\label{lm:correction2}
The sequence $\seq{g_{\bar{\theta}_\alpha}}{\alpha \in I}$ witnesses in $V[\P_\delta]$ the following: \beq \label{eq:cor2} \tcf (\prod_{\alpha<\kappa}\theta_\alpha,<_{\Jbd})^{V[\P_\delta]} = \cf(\delta).\eeq 
\end{lemma}

\begin{proof}
By Lemma \ref{lm:correction1}, the family $\seq{g_{\bar{\theta}_\alpha}}{\alpha \in I}$ is actually $<_{\Jbd}$-increasing and for every $f \in (\prod_{\alpha < \kappa}\theta_\alpha)^{V[\P_\delta]}$ in $V[\P_\delta]$ there is some $\alpha$ with $f <_{\Jbd} g_{\bar{\theta}_\alpha}$.
\end{proof}

\cite[Lemma 3.1 and Main Claim 3.3]{GS:pol} argue that this configuration is preserved under Prikry forcing $\Q_U$. Let us give more details:

\begin{lemma}\label{lm:correction3}
Suppose $\bar{\kappa} = \seq{\kappa_\alpha}{\alpha < \kappa}$ is a sequence of measurable cardinals converging to $\kappa$ such that $\GCH$ holds at the $\kappa_\alpha$'s. Let $\bar{\kappa}^+$ denote the sequence of the successors of the $\kappa_\alpha$'s. Then:
\bce[(i)]
\item In $V[\P_\delta]$, \beq \tcf (\prod_{\alpha<\kappa}\kappa_\alpha, <_{\Jbd})= \tcf(\prod_{\alpha<\kappa}\kappa^+_\alpha,<_{\Jbd}) = \cf(\delta).\eeq
\item In $V[\P_\delta]$, let $\Q_U$ denote the Prikry forcing defined with respect to some normal measure $U$ and let $\seq{\alpha_n}{n <\omega}$ denote a Prikry sequence which is included in $\bar{\kappa}$. Then the following hold:
\begin{multline} \label{m:1} \cf(\delta) = \tcf(\prod_{\alpha<\kappa}\kappa_\alpha,<_{\Jbd})^{V[\P_\delta]} = \tcf(\prod_{\alpha<\kappa}\kappa^+_\alpha,<_{\Jbd})^{V[\P_\delta]} = \\ \tcf(\prod_{\alpha<\kappa}\kappa_\alpha,<_{U^*})^{V[\P_\delta]} = \tcf(\prod_{\alpha<\kappa}\kappa^+_\alpha,<_{U^*})^{V[\P_\delta]} = \\ \tcf(\prod_{n<\omega}\alpha_n,<_{E^*})^{V[\P_\delta *\Q_U]} = \tcf(\prod_{n<\omega}\alpha_n^+,<_{E^*})^{V[\P_\delta *\Q_U]},\end{multline}
where $E$ is an ultrafilter extending $\Jbdpar{\omega}$, and $E^*$ its dual (and $U^*$ is the dual of $U$).
\ece
\end{lemma}

\begin{proof}
(i) follows directly from Lemma \ref{lm:correction2} in the present paper.

(ii) The first line in (\ref{m:1}) follows by Lemma \ref{lm:correction2} in the present paper, the second line by the fact that $U^*$ extends $\Jbd$ and true cofinalities are preserved when ideals are extended, and the last line follows by \cite[Lemma 3.1]{GS:pol}.
\end{proof}

The small value of $\uu$ is then obtained as a follow-up of \cite{GS:pol} in \cite{GS:small}, with \cite[Theorem 1.4]{GS:small} being the key ingredient. For concreteness let us assume that $\cf(\delta) = \kappa^+$. As we reviewed above, if $E$ is any ultrafilter extending the dual of $\Jbdpar{\omega}$ (and $E^*$ the dual prime ideal), we have in $V[\P_\delta * \Q_U]$: \beq \tcf(\prod_{n<\omega}\alpha_n,<_{E^*})^{V[\P_\delta *\Q_U]} = \tcf(\prod_{n<\omega}\alpha_n^+,<_{E^*})^{V[\P_\delta *\Q_U]} = \kappa^+.\eeq

Let us restate \cite[Theorem 1.4]{GS:small} with specific paramaters for our case:
\begin{theorem}[\cite{GS:small}]\label{GS:small:here}
Assume that:
\bce[(i)]
\item $\kappa$ is a singular strong limit cardinal with countable cofinality.
\item $E$ is a uniform ultrafilter on $\omega$ and $E^*$ its dual.
\item $\bar{\kappa} = \seq{\kappa_n}{n<\omega}$ is a sequence of regular cardinals converging to $\kappa$.
\item $U_n$ is a uniform ultrafiter on $\kappa_n$ for each $n<\omega$.
\item For every $n<\omega$ there is a $\sub^*$-decreasing sequence $\seq{A_{n,\alpha}}{\alpha<\theta_n}$ for some $\theta_n$ which generates $U_n$ (let $\bar{\theta} = \seq{\theta_n}{n<\omega}$).
\item $\chi_{\bar{\kappa}} = \tcf(\prod_{n<\omega}\kappa_n,<_{E^*})$, $\chi_{\bar{\theta}} = \tcf(\prod_{n<\omega}\theta_n,<_{E^*})$.
\ece
Then $\uu \le \chi_{\bar{\kappa}} \cdot \chi_{\bar{\theta}}$.
\end{theorem}

By setting $\bar{\kappa} = \seq{\alpha_n}{n<\omega}$ and $\bar{\theta} = \seq{\alpha^+_n}{n<\omega}$ in Theorem \ref{GS:small:here} and realizing that $\alpha_n$ is measurable in $V[\P_\delta *\Q_U]$ and $2^{\alpha_n} = \alpha_n^+$ so that \emph{(v)} holds, we have the desired \beq \uu = \kappa^+.\eeq

\brm
The iteration in \cite{GS:pol} has a singular cardinal length $\delta > \cf(\delta)>\kappa$ with $2^\kappa = \delta$ at the end. However, we can just as easily have $2^\kappa = \mu$ for some regular $\mu$ if we iterate up to an ordinal $\delta \in (\mu,\mu^+)$ with cofinality $\kappa^+$ (or greater); then $2^\kappa = \mu$ and the required pcf properties hold in $V[\P_\delta]$.
\erm

\section{Countable cofinality}\label{sec:countable}

In this section we obtain a singular strong limit cardinal $\kappa$ with countable cofinality on which $\SCH$ fails, $\uu = \kappa^{+}$ and the tree property, stationary reflection and the failure of approachability hold at $\kappa^{++}$.

\subsection{Definition of forcing}

Let $\kappa$ be a Laver-indestructible supercompact cardinal in the sense of the discussion in paragraph (\ref{PL}) and $\lambda$ a weakly compact cardinal above $\kappa$. Let us fix some $\delta \in (\lambda, \lambda^+)$ of cofinality $\kappa^+$.\footnote{Since we will have $2^\kappa = \kappa^{++}$ in the final model, $\kappa^+$ is the only interesting value of $\uu$. But in principle, the forcing can be iterated up to any ordinal with cofinality $>\kappa$.} Let $\P_\delta = \seq{(\P_\xi,\dot{Q}_\xi)}{\xi<\delta}$ be as in \cite{GS:pol}, iterated up to $\delta$: the forcing $\P_\delta$ is a $<\kappa$-supported iteration of the $\bar{\theta}$-dominating forcing $\Qt$ reviewed above, where $\bar{\theta}$'s are chosen so that every $\bar{\theta} \in V$ appears cofinally often below $\delta$. In particular the appropriate form of (\ref{eq:cor2}) holds in $V[\P_\delta]$ for every $\bar{\theta}$ in $V$: \beq \label{eq:gap2a} \tcf (\prod_{\alpha<\kappa}\theta_\alpha,<_{\Jbd})^{V[\P_\delta]} = \kappa^+.\eeq

By results in \cite{GS:pol}, $\P_\delta$ is $\kappa^+$-cc (in fact $\kappa^+$-Knaster) and preserves the supercompactness of $\kappa$, using the preparation $\P^L$ (see the paragraph related to (\ref{PL}) for more details).

Let us now define a forcing which will ensure the compactness principles at $\kappa^{++}$ while also ensuring small $\uu$. It is a version of the Mitchell forcing: it differs from the original Mitchell forcing (denoted $\M(\kappa,\delta)$) in the use of the forcing $\P_\delta$ instead of the Cohen forcing $\Add(\kappa,\delta)$.

\begin{definition}\label{def:Pbasic-s} $\P^*_\delta$ is a forcing with conditions $p = (p^0,p^1)$ such that:
\begin{itemize}
\item $p^0 \in \P_\delta$,
\item $p^1$ is a function with domain $\dom{p^1}$ of size at most $\kappa$ such that 
\beq \label{dom} \dom{p^1} \mbox{ is included in the set of successor cardinals below $\lambda$.}\eeq
\end{itemize}
The ordering is the usual Mitchell ordering: $(p^0,p^1) \le (p'^{0},p'^{1})$ iff $p^0 \le_{\P_\delta} p'^{0}$ and the domain of $p^1$ extends the domain of $p'^{1}$ and for all  $\alpha \in \dom{p'^{1}}$, $$\restr{p^0}{\alpha} \Vdash_{\P_\alpha} p^1(\alpha) \le p'^{1}(\alpha).$$
\end{definition}

\brm \label{rm:dom}

Let us say a few words about $\P^*_\delta$ and some alternative definitions. To show compactness principles in $V[\P^*_\delta]$ and $V[\P^*_\delta * \dot{\Q}_U]$, the ``hands-on'' method usually lifts a certain embedding with critical point $\lambda$ and gives an argument about the resulting quotient forcings. To simplify the quotient argument, it is desirable that $\P^*_\delta$ is uniform in the sense that at many places below $\lambda$ it looks like the tail segment of $\P_\delta$ in interval $[\lambda, \delta)$ (condition (\ref{dom})). However, in view of the preservation theorems for the tree property (see \cite{HS:ind}, and its modification here in Lemma \ref{tp}) and stationary reflection (see Theorem \ref{th:sr}), it is actually not necessary to prepare for the interval $[\lambda, \delta)$ below $\lambda$ because we can deal directly with the forcing $\P_{[\lambda,\delta)}*\dot{\Q}_U$. In particular if we changed the definition so that $\P_\delta$ is the Cohen forcing $\Add(\kappa,\lambda)$ up to $\lambda$ and then continues like $\P_{[\lambda,\delta)}$, we would still get the tree property and stationary reflection (together with small $\uu$) -- it is sufficient to apply the preservation theorems in \cite{HS:ind} and Theorem \ref{th:sr} over the Mitchell model $V[\M(\kappa,\lambda)]$ to the forcing $\P_{[\lambda,\delta)}*\dot{\Q}_U$. However, the preparation seems to be necessary for the failure of approachability because the preservation theorem for non-approachability is only formulated for centered forcings (and $\P_{[\lambda,\delta)}$ is not $\kappa$-centered).
\erm

For any $\gamma \le \delta$, we denote by $\P^*_\gamma$ and $\P_\gamma$ the natural initial stages of the forcing.

One can show that the usual product analysis in \cite{ABR:tree} applies to $\P^*_\delta$: for every ordinal $\gamma \le \delta$ (in fact, only ordinals of cofinality at least $\kappa^+$ are interesting for us) there are projections \beq \label{sigma} \pi_\gamma: \P^*_\gamma \to \P_\gamma \mbox{ and } \sigma_\gamma: \P_\gamma \x \T_\gamma \to \P^*_\gamma,\eeq where $\T_\gamma = \set{(1,p^1)}{(1,p^1) \in \P^*_\gamma}$ is a term forcing which is $\kappa^+$-closed in $V$. This analysis carries over to quotients:  if $G_\gamma$ is $\P^*_\gamma$-generic for an inaccessible $\gamma < \lambda$, and $G^0_\gamma$ is the $\P_\gamma$-generic derived by means of $\pi_\gamma$, then there are projections \beq \label{sigma*} \pi_{\gamma,\delta}:\P^*_\delta/G_\gamma \to \P_\delta/G^0_\gamma \mbox{ and } \sigma_{\gamma,\delta}: \P_\delta/G^0_\gamma \x \T_{\gamma,\delta} \to \P^*_\delta/G_\gamma,\eeq where $\T_{\gamma,\delta} = \set{(1,p^1)}{(1,p^1) \in \P^*_\delta/G_\gamma}$ is a term forcing $\kappa^+$-closed in $V[G_\gamma]$. 

In particular \beq \label{eq:Q} \P^*_\delta \mbox{ is forcing equivalent to } \P_\delta * \dot{R},\eeq for some $\dot{R}$ which is forced to be $\kappa^+$-distributive. By standard arguments, $\P^*_\delta$ collapses cardinals exactly in the interval $(\kappa^+,\Ups)$ and makes $2^\kappa = \kappa^{++} = \Ups$.

The first two lines in (\ref{m:1}) hold in $V[\P_\delta]$ for any normal measure $U$ on $\kappa$ in $V[\P_\delta]$, but they also hold in $V[\P^*_\delta]$ because the quotient $\dot{R}$ in (\ref{eq:Q}) is $\kappa^+$-distributive a hence it does not add new functions into the product $\prod_{\alpha<\kappa}\theta_\alpha$ or new subsets of $\kappa$ ($U$ therefore remains a normal measure in $V[\P^*_\delta]$):

\begin{lemma}\label{lm:key1}
For $U$ as in the previous paragraph and every $\bar{\theta}$ in $V$, \beq \kappa^+ = \tcf (\prod_{\alpha<\kappa}\theta_\alpha,<_{\Jbd})^{V[\P^*_\delta]} = \tcf (\prod_{\alpha<\kappa}\theta_\alpha,<_{U^*})^{V[\P^*_\delta]}.\eeq
\end{lemma}

\begin{proof}
By (\ref{eq:Q}), $\dot{R}$ does not add new functions to $\prod_{\alpha<\kappa}\theta_\alpha$, and so every sequence witnessing the true cofinality in $V[\P_\delta]$ witnesses the same fact in $V[\P^*_\delta]$.
\end{proof}

Let $U$ be any normal measure in $V[\P_\delta]$ on $\kappa$, and let $\dot{\Q}_U$ be the Prikry forcing defined with respect to $U$ in $V[\P^*_\delta]$ (note that $\Q_U$ is an element of $V[\P_\delta]$).

\begin{definition}\label{def:P}
Let $\R$ denote the forcing $\P^*_\delta * \dot{\Q}_U$.
\end{definition}

We are going the prove the following theorem:

\begin{theorem}\label{th:countable}
Let $\kappa$ be Laver-indestructible in the sense of (\ref{PL}) and $\Ups$ a weakly compacty cardinal above $\kappa$. Then in $V[\R]$, $2^\kappa = \kappa^{++} = \Ups$, $\kappa$ is a singular strong limit cardinal of countable cofinality and the following hold:
\bce[(i)]
\item $\uu = \kappa^+$.
\item $\TP(\kappa^{++})$, $\SR(\kappa^{++})$ and $\neg \AP(\kappa^{++})$.
\ece
\end{theorem}

The proof will be given in a sequence of lemmas in three subsections.

\subsection{Small ultrafilter number}

\begin{lemma}\label{4.5}
$\uu = \kappa^+$ holds in $V[\R]$.
\end{lemma}

\begin{proof}
This follows the same way as for $V[\P_\delta][\dot{\Q}_U]$ in \cite{GS:pol,GS:small}: by Lemma \ref{lm:key1}, the relevant pcf facts are still true in $V[\P^*_\delta]$, and so \cite[Lemma 3.1]{GS:pol} can be applied over $V[\P^*_\delta]$. The result follows by \cite[Theorem 1.4]{GS:small} -- which we reviewed in Theorem \ref{GS:small:here} -- applied in $V[\R] = V[\P^*_\delta][\dot{\Q}_U]$ (notice that Theorem \ref{GS:small:here} is a $\ZFC$ theorem).
\end{proof}

\subsection{Tree property}\label{sec:tp}

In \cite{HS:ind}, we proved that over the Mitchell model, the tree property at $\lambda$ is indestructible under all $\kappa^+$-cc forcings living in the intermediate model $V[\Add(\kappa,\lambda)]$.

We indicate how to modify the argument from \cite{HS:ind} to be applicable in the present context.

\begin{lemma}\label{tp}
$\TP(\kappa^{++})$ holds in $V[\R]$.
\end{lemma}

\begin{proof}
Let us fix a $\P_\lambda$-name $\dot{\Q}$ for the two-step iteration $\P_{[\lambda,\delta)} * \dot{\Q}_U$ over $V[\P^*_\lambda]$ (this can be done because $\P_{[\lambda,\delta)} * \dot{\Q}_U$ is already added by $\P_\lambda$ because we assume that $U$, and hence $\Q_U$, are elements of $V[\P_\delta]$).

We will show that the $\kappa^+$-cc of $\dot{\Q}$ is enough to argue that \beq \label{equiv} \P^*_\delta * \dot{\Q}_U \equiv \P^*_\lambda * \dot{\Q} \eeq forces the tree property, without using any other properties of the Prikry forcing $\dot{\Q}_U$.

Suppose for contradiction there is a $\lambda$-Aronszajn tree $T$ in $V[\R]$ and let us fix a name $\dot{T}$ for $T$ (we view $T$ a tree on $\lambda$). Let us fix a weakly compact embedding $j: M \to N$ with critical point $\lambda$ such that $M$ has size $\lambda$, is closed under $<\lambda$-sequences and contains all relevant parameters, in particular $\dot{T}$ and the forcing $(\P_\lambda \x \T_\lambda) * \dot{\Q}$. We can further assume that $j$ itself is an element of $N$ which implies -- by the $\lambda$-cc of $(\P_\lambda \x \T_\lambda) * \dot{\Q}$ -- that $\restr{j}{(\P_\lambda \x \T_\lambda) * \dot{\Q})}$ is a regular embedding in $N$.

By (\ref{sigma}), we know there is a projection $\sigma_\lambda: \P_\lambda \x \T_\lambda \to \P^*_\lambda$. Let $(\tilde{G}^0 \x \tilde{G}^1) * h^*$ be $j((\P_\lambda \x \T_\lambda) * \dot{\Q})$-generic filter over $V$ (and hence also over $N$). Since $(\P_\lambda \x \T_\lambda) * \dot{\Q}$ is $\lambda$-cc, $j$ restricted to $(\P_\lambda \x \T_\lambda) * \dot{\Q}$ sends maximal antichains to maximal antichains, and therefore $(\tilde{G}^0 \x \tilde{G}^1) * h^*$ generates an $M$-generic filter $(G^0 \x G^1) * h$ for $(\P_\lambda \x \T_\lambda) * \dot{\Q}$; furthermore, $\tilde{G}^0 \x \tilde{G}^1$ generates an $N$-generic filter $G^*$ for $j(\P^*_\lambda)$ and an $M$-generic filter $G$ for $\P^*_\lambda$ such that $j$ lifts in $V[(\tilde{G}^0 \x \tilde{G}^1)*h^*]$ to: \beq \label{eq:decompose} j: M[G][h] \to N[G^* * h^*] = N[G][h][G_Q],\eeq where $G_Q$ is a generic filter for the quotient $Q = j(\P^*_\lambda*\dot{\Q})/G*h$. 

We will to show that over $N[G][h]$, \beq \label{eq:embed}\mbox{there is a projection onto } Q \mbox{ from } j(\P_\lambda * \dot{\Q}))/(G^0 * h) \x \T_{\lambda,j(\lambda)},\eeq where $\T_{\lambda,j(\lambda)}$ is the term forcing of $j(\P^*_\lambda)/G$ (it is composed of conditions of the form $(1_{j(\P_\lambda)},p^1)$). We will further show that $j(\P_\lambda * \dot{\Q})/(G^0 * h)$ is $\kappa^+$-cc over $N[G][h]$ and $\T_{\lambda,j(\lambda)}$ is $\kappa^+$-closed in $N[G]$ which will allow us to finish the argument as in \cite{HS:ind}.

Since $j$ is the identity on the conditions in $G$, we have \beq j''(G*h) = \set{(p,j(\dot{q}))}{p \in G \mbox{ and }\dot{q}^G \in h}.\eeq

Let us write explicitly the relevant quotients we are going to use: \begin{multline} Q = \set{(p^*,\dot{q}^*) \in j(\P^*_\lambda * \dot{\Q})}{N[G][h] \models \mbox{``}(p^*,\dot{q}^*) \\ \mbox{ is compatible with }j''(G*h)\mbox{''}},\end{multline} where we can assume that $\dot{q}^*$ depends by elementarity only on $j(\P_\lambda)$. Further, \begin{multline} j(\P_\lambda * \dot{\Q})/(G^0 * h) = \set{(p^{*0},\dot{q}^*) \in j(\P_\lambda * \dot{\Q})}{N[G^0][h] \models \\ \mbox{``}(p^{*0},\dot{q}^*) \mbox{ is compatible with } j''(G^0 * h)\mbox{''}}.\end{multline} Lastly, \begin{multline} \T_{\lambda,j(\lambda)} = \{(1_{j(\P_\lambda)},p^{*1}) \;| \\ N[G] \models \mbox{``}(1_{j(\P_\lambda)},p^{*1}) \mbox{ is compatible with }j''G = G\mbox{''}\}.\end{multline}

Let us define a function $\pi: j(\P_\lambda * \dot{\Q}))/(G^0 * h) \x \T_{\lambda,j(\lambda)} \to Q$ by \beq \pi((p^{*0},\dot{q}^*),p^{*1}) = (p^*,\dot{q}^*),\eeq where $p^* = (p^{*0},p^{*1})$.

\begin{Claim}\label{pi}
$\pi$ is a projection from $j(\P_\lambda * \dot{\Q})/G^0 *h \x \T_{\lambda,j(\lambda)}$ onto $Q$.
\end{Claim}

\begin{proof}
First notice that $\pi$ is correctly defined: if $(p^{*0},\dot{q}^*)$ is compatible with $j''(G^0 *h)$, and $(1_{j(\P_\lambda)},p^{*1})$ is compatible with $G$, then $(p^*,\dot{q}^*)$ is compatible with $j''(G*h)$.

If $((p^{*0},\dot{q}^*),p^{*1}) \le ((r^{*0},\dot{s}^*),r^{*1})$, then clearly $p^* \le r^*$; moreover, $p^{*0} \Vdash \dot{q}^* \le \dot{s}^*$ implies $p^* \Vdash \dot{q}^* \le \dot{s}^*$ because $p^* = (p^{*0},p^{*1})$. It follows $(p^*,\dot{q}^*) \le (r^*, \dot{s}^*)$, and hence $\pi$ is order-preserving.

Suppose now $(p^*,\dot{q}^*) \le \pi((r^{*0},\dot{s}^*),r^{*1})= (r^*,\dot{s}^*)$ are given. We wish to find a condition extending $((r^{*0},\dot{s}^*),r^{*1})$ whose $\pi$-image extends $(p^*,\dot{q}^*)$. First notice that $p^* \Vdash \dot{q}^* \le \dot{s}^*$ implies $p^{*0} \Vdash \dot{q}^* \le \dot{s}^*$ because of our convention that $\dot{q}^*$ and $\dot{s}^*$ depend only on $\P_\lambda$. Now we use a standard trick with names: Consider conditions $(p^{*0}, \dot{q}^*)$ and $p^{*1'}$ where the name $p^{*1'}$ interprets as $p^{*1}$ below $p^{*0}$ and as $r^{*1}$ otherwise; then $((p^{*0},\dot{q}^*),p^{*1'})$ is as required.
\end{proof}

Finally, we need the following Claim:

\begin{Claim}\label{claim:cc}
\bce[(i)]
\item $\T_{\lambda,j(\lambda)}$ is $\kappa^+$-closed in $N[G]$.
\item $j(\P_\lambda * \dot{\Q})/G^0 *h$ is $\kappa^+$-cc over $N[G][h]$.
\item $\dot{\Q}^{G^0} * j(\P_\lambda * \dot{\Q})/G^0 * \dot{h}$ is $\kappa^+$-cc over $N[G]$, where $j(\P_\lambda * \dot{\Q})/G^0 * \dot{h}$ denotes a $\dot{\Q}^{G^0}$-name for the quotient.
\ece
\end{Claim}

\begin{proof}
(i) This a standard fact (see for instance \cite{ABR:tree}).

(ii) By elementarity, \beq j(\P_\lambda * \dot{\Q}) \mbox{ is $\kappa^+$-cc over $N$.}\eeq The term forcing $\T_\lambda$ is $\kappa^+$-closed over $N$. By Easton's lemma \beq \label{eq:ccc} j(\P_\lambda*\dot{\Q}) \mbox{ is $\kappa^+$-cc over $N[G^1]$.}\eeq Since $j$ restricted to $\P_\lambda * \dot{\Q}$ is a regular embedding, $j(\P_\lambda * \dot{\Q})$ factors over $N$ (and then also over $N[G^1]$) as \beq (\P_\lambda * \dot{\Q}) * j(\P_\lambda * \dot{\Q})/\dot{G}^0 * \dot{h},\eeq  where $j(\P_\lambda * \dot{\Q})/\dot{G}^0 * \dot{h}$ is an $\P_\lambda * \dot{\Q}$-name for the quotient. It follows by (\ref{eq:ccc}), and properties of two-step iterations, that over $N[G^1]$, the $\kappa^+$-cc forcing $\P_\lambda * \dot{\Q}$ forces that $ j(\P_\lambda * \dot{\Q})/\dot{G}^0 * \dot{h}$ is $\kappa^+$-cc. In particular, $j(\P_\lambda * \dot{\Q})/G^0 *h$ is $\kappa^+$-cc over $N[G^1][G^0*h]$.

Since there is a natural projection from $(\P_\lambda *\dot{\Q}) \x \T_\lambda$ onto $\P^*_\lambda * \dot{\Q}$ (analogously to the projection $\pi$ mentioned above), it follows that $j(\P_\lambda * \dot{\Q})/G^0 *h$ is $\kappa^+$-cc over $N[G][h]$ as desired (since the chain condition is preserved downwards).

(iii) Recall that $\dot{\Q}^{G^0}$ is $\kappa^+$-cc in $N[G]$ by our initial assumptions. By (ii) of the present Claim, $\dot{\Q}^{G^0}$ forces over $N[G]$ that $j(\P_\lambda * \dot{\Q})/G^0 *\dot{h}$ is $\kappa^+$-cc. By general forcing properties this means the two-step iteration $\dot{\Q}^{G^0} * j(\P_\lambda * \dot{\Q})/G^0 * \dot{h}$ is $\kappa^+$-cc in $N[G]$.
\end{proof}

We assume for contradiction there is in $M[G][h]$ a $\lambda$-Aronszajn tree $T$. By standard arguments, we can assume that $T$ is also in $N[G][h]$ (and is Aronszajn here), and $T$ has a cofinal branch in $N[G][h][G_Q]$ because of the lifted embedding $j$ in (\ref{eq:decompose}). We will argue that the forcing $Q$ cannot add a cofinal branch to $T$, which is a contradiction.

Working over $N[G][h]$, $\lambda = (\kappa^{++})^{N[G][h]}$ and therefore by Fact \ref{Kunen}(ii) and Claim \ref{claim:cc}(ii), $j(\P_\lambda * \dot{\Q})/G^0 *h$ cannot add a cofinal branch to the $\lambda$-Aronszajn tree $T$. Using the fact that $2^\kappa = \lambda$ in $N[G]$, and Fact \ref{f:spencer} applied over $N[G]$ to the $\kappa^+$-closed forcing $\T_{\lambda,j(\lambda)}$ and to the $\kappa^+$-cc forcing $\dot{\Q}^{G^0} * j(\P_\lambda * \dot{\Q})/G^0 * \dot{h}$ (see Claim \ref{claim:cc}(iii)), it follows that $\T_{\lambda,j(\lambda)}$ cannot add a cofinal branch to $T$ over a generic extension of $N[G][h]$ by the quotient $j(\P_\lambda * \dot{\Q})/G^0 *h$. Thus, the product \beq \label{product} \T_{\lambda,j(\lambda)} \x j(\P_\lambda * \dot{\Q})/G^0 *h \eeq does not add cofinal branches to $T$ over $N[G][h]$. However, by Claim \ref{pi}, there is a projection onto the quotient $Q$ from the product (\ref{product}), and therefore $Q$ cannot add a cofinal branch to $T$. It follows that $T$ has no cofinal branch in $N[G][h][G_Q]$ which is the desired contradiction.

This ends the proof of Lemma \ref{tp}.
\end{proof}

\subsection{Stationary reflection}\label{sec:sr}

It is standard to show that stationary reflection holds at $\kappa^{++}$ in $V[\P^*_\lambda]$, so it remains to show that the $\kappa^+$-cc forcing $\P_{[\lambda,\delta)} * \dot{\Q}_U$ preserves $\SR(\kappa^{++})$ over $V[\P^*_\lambda]$. This follows from Theorem \ref{th:sr} with $\lambda = \kappa^+$.

\brm
We should stress that unlike the tree property argument, we do not need that $\P_{[\lambda,\delta)} * \dot{\Q}_U$ should live already in $V[\P_\lambda]$ -- the reason is that the preservation Theorem \ref{th:sr} for stationary reflection is much stronger than the preservation theorem for the tree property.
\erm

\begin{theorem}\label{th:sr}
Suppose $\lambda$ is a regular cardinal, $\SR(\lambda^+)$ holds and $\Q$ is $\lambda$-cc. Then $\SR(\lambda^+)$ holds in $V[\Q]$.
\end{theorem}

\begin{proof}
Suppose for contradiction there are $p_0 \in \Q$ and $\dot{S}$ such that $p_0$ forces that $\dot{S}$ is a non-reflecting stationary subset of $\lambda^+ \cap \cof(<\lambda)$. Set \beq \label{U} U_{p_0} = \set{\gamma \in \lambda^+ \cap \cof(<\lambda)}{\exists p \le p_0 \; p \Vdash \gamma \in \dot{S}}.\eeq $U_{p_0}$ is a stationary set: for every club $C \sub \lambda^+$, $p_0$ forces  $C \cap \dot{S} \neq \emptyset$, and because $p_0$ also forces $\dot{S} \sub U_{p_0}$, it forces $C \cap U_{p_0} \neq \emptyset$, which is equivalent to $C \cap U_{p_0}$ being non-empty in $V$.  By $\SR(\lambda^+)$ there is some $\alpha < \lambda^+$ of cofinality $\lambda$ such that \beq \label{s:1} \mbox{$U_{p_0} \cap \alpha$ is stationary}.\eeq By our assumption \beq \label{s:2}p_0 \Vdash \dot{S} \cap \alpha \mbox{ is non-stationary}.\eeq We will argue that (\ref{s:1}) and (\ref{s:2}) are contradictory, which will finish the proof.

First recall that by the $\lambda$-cc of $\Q$, every club subset of an ordinal $\alpha$ of cofinality $\lambda$ in $V[\Q]$ contains a club in the ground model. It follows by (\ref{s:2}) that there is a maximal antichain $A$ below $p_0$ such that for every $p \in A$ there is some club $D$ in $\alpha$ in the ground model with $p \Vdash \dot{S} \cap D = \emptyset$. Let us fix for each $p \in A$ some $D_p$ such that $p \Vdash \dot{S} \cap D_p = \emptyset$.

Set \beq C = \bigcap \set{D_p}{p \in A}.\eeq $C$ is a club subset of $\alpha$ because $A$ has size $<\lambda$ and $\alpha$ has cofinality $\lambda$. It holds \beq \label{alt} p_0 \Vdash \dot{S} \cap C = \emptyset \eeq because conditions forcing $\dot{S} \cap C = \emptyset$ are dense below $p_0$: for every $q \le p_0$ there is some $p \in A$ which is compatible with $q$, and any $r \le p,q$ forces $\dot{S} \cap D_p = \emptyset$. Since $C \sub D_p$, this implies $r \le q$ forces $\dot{S} \cap C = \emptyset$.

However, by (\ref{s:1}) there must be  $\gamma \in C \cap U_{p_0} \cap \alpha$, and therefore some $p \le p_0$ such that $p \Vdash \gamma \in \dot{S} \cap C$. This contradicts (\ref{alt}).
\end{proof}

Let us add a note of general interest. For $S \sub \lambda^+$, let $\Ref(S)$ denote the set of points of cofinality $\lambda$ on which $S$ reflects. A little analysis of the proof of Theorem \ref{th:sr} shows that for every $\alpha \in \Ref(U_{p_0})$, \beq \label{not} p_0 \not \Vdash \dot{S} \cap \alpha \mbox{ is non-stationary.}\eeq This is sufficient to argue that $\Q$ preserves $\SR(\lambda^+)$, but it does not seem in general sufficient for preservation of some stronger forms of stationary reflection. Recall that \emph{club stationary reflection} at $\lambda^+$, $\CSR(\lambda^+)$, says that for every stationary $S \sub \lambda^+ \cap \cof(<\lambda)$ there is a club $C \sub \lambda^+$ such that $C \cap \cof(\lambda) \sub \Ref(S)$ (we say that $\Ref(S)$ contains a $\lambda$-club). See \cite{m:sr} and \cite{JS:full} for more information about this concept. With more care, we can show something about preservation of $\CSR(\lambda^+)$ as well.

\begin{theorem}\label{th:csr}
Suppose $\lambda$ is a regular cardinal and $\CSR(\lambda^+)$ holds. Then the following hold:
\bce[(i)]
\item If $\Q$ is a $\lambda$-cc forcing with a dense subset $D$ of size $\le \lambda$, then $\Q$ preserves $\CSR(\lambda^+)$.
\item Suppose $\lambda = \kappa^+$ and $\kappa^{<\kappa} = \kappa$. Then $\Add(\kappa,\alpha)$ preserves $\CSR(\lambda^+)$ for any $\alpha$.
\item Suppose $\lambda = \kappa^+$ and $\kappa$ is measurable ($2^\kappa$ can have any value). Then $\Q_U$, the Prikry forcing with respect to a normal measure $U$ on $\kappa$, preserves $\CSR(\lambda^+)$.\footnote{The proof of (iii) was suggested to us by Menachem Magidor.}
\ece
\end{theorem}

\begin{proof}
(i). Suppose for simplicity that $1_\Q$ forces that $\dot{S}$ is a stationary subset of $\lambda^+$. For every $p \in D$, consider $U_p$ defined as in (\ref{U}). By our assumption every $\Ref(U_p)$ contains a $\lambda$-club, and therefore \beq \bigcap_{p \in D}\Ref(U_p) \mbox{ contains a $\lambda$-club $C$.}\eeq We will show that $1_\Q$ forces that $\dot{S}$ reflects on every point in $C$. Let $\gamma$ be in $C$ and suppose for contradiction that there is $p$ which forces $\dot{S} \cap \gamma$ is disjoint from some club $C^*$ in $\gamma$. Choose $p^* \le p$ in $D$: then $p^*$ does not force that $\dot{S} \cap \gamma$ is non-stationary by (\ref{not}) because $\gamma \in \Ref(U_{p^*})$, but it also forces that it avoids $C^*$, contradiction.

(ii). By the homogeneity of Cohen forcing, it suffices to show that $\Q = \Add(\kappa,\lambda^+)$ preserves $\CSR(\lambda^+)$. Suppose for simplicity that $1_\Q$ forces that $\dot{S}$ is a stationary subset of $\lambda^+$. Let $\seq{p_\alpha}{\alpha< \lambda^+}$ be some enumeration of $\Q$. Then there is a club $C_\Q$ in $\lambda^+$ such that for every $\gamma \in C_\Q$ of cofinality $\lambda$, every condition with its domain included in $\gamma$ appears as $p_\delta$ for some $\delta<\gamma$. For each $\alpha<\lambda^+$, let $C_\alpha$ denote a $\lambda$-club contained in $\Ref(U_{p_\alpha})$ with the property that 

$(*)$ for every $\gamma \in C_\alpha$ of cofinality $\lambda$, if some condition $p \le p_\alpha$ forces $\xi \in \dot{S} \cap \gamma$, then there is some $p' \le p_\alpha$ which forces $\xi \in \dot{S} \cap \gamma$ such that $\dom{p'} \sub \gamma$.

By our assumption \beq C= C_\Q \cap \triangle_{\alpha <\lambda^+}C_\alpha \mbox{ is a $\lambda$-club,}\eeq where $\triangle$ denotes the diagonal intersection. We show that $1_\Q$ forces that $\dot{S}$ reflects on every element in $C$. Let $\gamma \in C$ be fixed. Suppose for contradiction that there is $p$ which forces that $\dot{S}$ is disjoint from some club $C^*$ in $\gamma$. Then $p$ restricted to $\gamma$ appears as $p_\delta$ for some $\delta < \gamma$, and there are $p' \le p_\delta$ and $\xi < \gamma$ with $p' \Vdash \xi \in \dot{S} \cap C^*$ (because $\gamma \in \Ref(U_{p_\delta})$). By $(*)$ we can assume that $p'$ has its domain included in $\gamma$, and so $p' \cup p$ is a valid condition extending $p$ which is a contradiction.

(iii) Let $\Q_U$ denote the Prikry forcing defined with respect to some normal measure $U$ on $\kappa$. Suppose $1_{\Q_U}$ forces that $\dot{S}$ is a stationary subset of $\kappa^{++}\cap \cof(<\kappa^+)$. We can assume that there is a fixed stem $s$ such that \beq U = \set{\alpha < \kappa^{++}}{\exists A_\alpha \; (s,A_\alpha) \Vdash \alpha \in \dot{S}}\eeq is stationary. For an and-extension $t \sqsupseteq s$, let us define \beq T_t = \set{\alpha \in U}{t \setminus s \sub A_\alpha}.\eeq We say that $T_t$ is \emph{good} if $T_t$ is stationary. Since there are only $\kappa$-many $t$'s, there must be some good $T_t$. We need the following: 

$(*)$ There exists a set $\Ah \in U$ such that if $t \sqsupseteq s$ and $t \setminus s \sub \Ah$, then $T_t$ is good.

Let us prove $(*)$. Denote $B = \set{\alpha < \kappa}{\mx{max}(s) < \alpha} \in U$. By Rowbottom's theorem, there is a homogeneous set $\Ah \in U$ for the partition of $[B]^{<\omega}$ into two colors: (a)  $x$ gets color 0 if $T_{s \cup x}$ is good, and (b) $x$ gets color $1$ if $T_{s \cup x}$ is not good. We want to show that $\Ah$ is homogeneous in color $0$. Suppose for contradiction that $\Ah$ is homogenous in color 1. First note that \begin{equation} U = \set{\alpha <\kappa^{++}}{(s,\Ah \cap A_\alpha) \Vdash \alpha \in \dot{S}}.\end{equation} For $t \sqsupseteq s$, let $T^*_t = \set{\alpha \in U}{t \setminus s \sub \Ah \cap A_\alpha}$ so that $T^*_t \sub T_t$. Since there are only $\kappa$-many such $t$'s and $U$ is stationary, there must be some $t$ such that $T^*_t$ is stationary. Let us choose such $t$. Since we assume that $\Ah$ is homogeneous in color $1$ and $t \setminus s \sub \Ah$, $T_t$ must be non-stationary. This contradicts our choice of $t$.

With $(*)$ we finish the argument as follows. By $\CSR(\kappa^{++})$, we can choose a $\kappa^+$-club $D$ on which every good $T_t$ reflects. We wish to argue that for every $\gamma \in D$, \begin{equation} (s,\Ah) \Vdash \dot{S} \cap \gamma \mbox{ is stationary}.\end{equation} Suppose this is not the case for some $\gamma \in D$. Then there is a club $C^* \sub \gamma$ and $(t,A) \le (s,\Ah)$ which forces that $C^* \cap \dot{S}$ is empty. By the homogeneity of $\Ah$, $T_t$ is good. By the choice of $D$, there is some $\xi < \gamma$ in $T_t \cap C^*$. Since $T_t \sub U$, $(s,A_\xi) \Vdash \xi \in \dot{S}$. By the definition of $T_t$, $t\setminus s \sub A_\xi$, and therefore $(t,A \cap A_\xi)$ is a condition which extends both $(s,A_\xi)$ and $(t,A)$ and forces $C^* \cap \dot{S}$ is non-empty. This is a contradiction.
\end{proof}

We think it is plausible that arguments as in Theorem \ref{th:csr} can be used to argue that it is consistent to have $\CSR(\kappa^{++})$ for a singular strong limit $\kappa$ with a small $\uu$. However, we do not know if a full analogue of Theorem \ref{th:sr} holds for $\CSR(\lambda^+)$.

\brm
Some preservation theorems for stationary reflection were known before: it was known that the Prikry-style forcings at $\kappa$ preserve $\SR(\kappa^{++})$ due to their Prikry property (see for instance \cite{8fold}) and that in general $\kappa^+$-cc forcings of size $<\kappa^{++}$ preserve $\SR(\kappa^{++})$ and the Cohen forcing at $\omega$ of any length preserves $\SR(\omega_2)$ (attributed to Neeman in \cite{KG8}).
\erm

\begin{lemma}\label{sr}
$\SR(\kappa^{++})$ holds in $V[\R]$.
\end{lemma}

\begin{proof}
This follows from Theorem \ref{th:sr} and the fact that $\SR(\kappa^{++})$ holds in $V[\P^*_\lambda]$.
\end{proof}

\subsection{Failure of approachability}\label{sec:ap}

Let us first argue that $\neg \AP(\kappa^{++})$ holds in $V[\P^*_\delta]$.

\begin{lemma}\label{ap-delta}
$\neg\AP(\kappa^{++})$ holds in $V[\P^*_\delta]$.
\end{lemma}

\begin{proof}
This is like the argument from \cite{8fold} with $\P_\delta$ instead of $\Add(\kappa,\delta)$. Clause (\ref{dom}) from Definition \ref{def:Pbasic-s} ensures that for every inaccessible $\alpha<\lambda$, the quotient forcing $\P^*_\delta/G_\alpha$, where $G_\alpha$ is $\P^*_\alpha$-generic, can be written as \beq \P_{[\alpha,\alpha^+)} * \P^*_{[\alpha^+,\delta)},\eeq where $\P_{[\alpha,\alpha^+)}$ adds new fresh subsets of $\kappa$ without introducing new collapses, which makes it possible to argue for the non-approachability (see \cite{8fold} for more details).
\end{proof}

We will use a theorem from \cite[Corollary 2.2.]{GK:a} to argue that the Prikry forcing $\Q_U$ preserves $\neg \AP(\kappa^{++})$ over $V[\P^*_\delta]$. We phrase the theorem in a way which is suitable for us (Gitik and Krueger use a different indexation for $\AP$):

\begin{theorem}[\cite{GK:a}]\label{th:Gitik}
Assume $\neg \AP(\kappa^{++})$ holds and $\Q$ is $\kappa$-centered. Then the forcing $\Q$ forces $\neg \AP(\kappa^{++})$.
\end{theorem}

\begin{lemma}\label{ap}
$\neg\AP(\kappa^{++})$ holds in $V[\R]$.
\end{lemma}

\begin{proof}
This follows by Lemma \ref{ap-delta} and Theorem \ref{th:Gitik} because $\Q_U$ is $\kappa$-centered over $V[\P^*_\delta]$ (see Section \ref{sec:Prk} for details).
\end{proof}

This ends the proof of Theorem \ref{th:countable}.

\section{Uncountable cofinalities}\label{sec:uncountable}

In their paper \cite{GS:pol}, Garti and Shelah prove their theorem for the countable cofinality using the Prikry forcing $\Q_U$, but mention in \cite[Remark 3.2]{GS:pol} that the same argument holds when the Prikry forcing is replaced with the Magidor forcing $\Qm$ from \cite{M:cof}. Since the combinatorial argument in \cite[Theorem 1.4]{GS:small} (reviewed as Theorem \ref{GS:small:here} in our paper for cofinality $\omega$) for the small ultrafilter number is not specific for the countable cofinality, we get:

\begin{theorem}\label{th:unc}
Let $\kappa$ be Laver-indestructible in the sense of (\ref{PL}) and $\Ups$ a weakly compact cardinal above $\kappa$ and let $\mu<\kappa$ be a fixed regular cardinal. Then there is a forcing notion $\Pm$ such that in $V[\Pm]$, $2^\kappa = \kappa^{++} = \Ups$, $\kappa$ is a singular strong limit cardinal of cofinality $\mu$ and the following hold:
\bce[(i)]
\item $\uu = \kappa^+$.
\item $\TP(\kappa^{++})$, $\SR(\kappa^{++})$ and $\neg \AP(\kappa^{++})$.
\ece
\end{theorem}

\begin{proof}
Let $\Pm$ be defined as $\R$ in Definition \ref{def:P} with the Magidor forcing $\Qm$ instead of the Prikry forcing $\Q_U$. $\Qm$ is defined with respect to a sequence $\vec{U}$ of length $\mu$ (see Section \ref{sec:Prk}) and we take care to choose $\vec{U}$ which is an element of $V[\P_\delta]$ ($\kappa$ is supercompact in $V[\P_\delta]$ by arguments in \cite{GS:pol}, reviewed in (\ref{PL}), so this is possible). Moreover, by Lemma \ref{lm:vec-U}, $\vec{U}$ is still a Mitchell sequence in $V[\P^*_\delta]$ because $\P^*_\delta$ factors as $\P_\delta* \dot{R}$ for a $\kappa^+$-distributive $\dot{R}$. As in (\ref{equiv}), we get: \beq \Pm = \P^*_\delta * \dot{\Q}^{\mathrm{Mag}}_{\vec{U}} \equiv \P^*_\lambda * \dot{\Q},\eeq where $\dot{\Q} = \P_{[\lambda,\delta)} * \Qm$ is forced to be $\kappa^+$-cc and where we assume that $\dot{\Q}$ is a $\P_\lambda$-name.

The fact that $\Pm$ forces $\uu = \kappa^+$ follows as in Lemma \ref{4.5}, taking into account \cite[Remark 3.2]{GS:pol}.

The proof that the tree property holds at $\kappa^{++}$ follows by the same argument as in Lemma \ref{tp} because it uses just the fact that the forcing is $\kappa^+$-cc and lives in $V[\P_\lambda]$. The arguments for stationary reflection and the failure of approachability follow by the indestructibility Theorems \ref{th:sr} and \ref{th:Gitik}.
\end{proof}

\section{Open questions}\label{sec:open}

We expect that with little more work, the value of $2^\kappa$ could be made arbitrarily large while keeping $\uu$ at a prescribed cardinal. However, we have not verified this in detail, so let us ask the following explicit question:

\begin{question}Is it consistent to have a singular strong limit cardinal $\kappa$ with $\uu = \nu$ for a regular $\kappa^+ < \nu<2^\kappa$ and with the tree property, stationary reflection and the failure of approachability at $\kappa^{++}$?
\end{question}

In \cite{GS:pol}, Garti and Shelah say it is plausible that the assumption of a supercompact cardinal can be weakened to that of a hypermeasurable cardinal. It would seem that in the present proof the role of supercompactness of $\kappa$ is just to make sure that $\kappa$ stays measurable (or sufficiently large) in $V[\P_\delta]$, and it is known that there are methods to ensure this from weaker hypotheses (see for instance \cite{RH:Laver}). So it is plausible that the following has an affirmative answer (especially in view of the recent result in \cite{GGS:u}):

\begin{question} Is the consistency of $\kappa$ being singular strong limit, $\uu = \kappa^+$, $2^\kappa>\kappa^+$ with compactness at $\kappa^{++}$ provable from a weaker hypothesis than a supercompact cardinal?
\end{question}

Can we bring the result down to $\aleph_\omega$ or $\aleph_{\omega_1}$? Note that the method of \cite{HS:ind} may not be directly applicable here because the Prikry forcing with collapses is not necessarily definable in the right intermediate model to apply the tree property preservation argument.

\begin{question} Is it consistent to have $\aleph_\omega$ strong limit with $2^{\aleph_\omega}>\aleph_{\omega+1}$, $\mathfrak{u}(\aleph_\omega) = \aleph_{\omega+1}$ and the tree property at $\aleph_{\omega+2}$? Can a similar result be obtained for $\aleph_{\omega_1}$?
\end{question}

It is also possible to consider compactness at $\kappa^+$, for instance:

\begin{question} Is it consistent to have a singular strong limit cardinal $\kappa$ with $2^\kappa>\kappa^+$, $\uu = \kappa^+$ and the tree property at $\kappa^+$?
\end{question}

\end{document}